\documentclass[12pt]{amsart}

\usepackage[OT2, T1]{fontenc}

\usepackage{amscd}
\usepackage{amsmath}
\usepackage{amssymb}
\usepackage{mathrsfs} 			
\usepackage{units}
\usepackage[all]{xy}

\usepackage{algorithm}
\usepackage{algorithmic}

\def\frk{\frak}               

\def\mm{{\frk m}}

\def\Phi{{\frk n}}
\def\Phi{{\frk N}}
\def\opn#1#2{\def#1{\operatorname{#2}}} 
\opn\chara{char} \opn\length{\ell} \opn\pd{pd} \opn\rk{rk}
\opn\projdim{proj\,dim} \opn\injdim{inj\,dim} \opn\rank{rank}
\opn\depth{depth} \opn\sdepth{sdepth} \opn\fdepth{fdepth}
\opn\grade{grade} \opn\height{height} \opn\embdim{emb\,dim}
\opn\codim{codim}  \opn\min{min} \opn\max{max}

\opn\Tr{Tr} \opn\bigrank{big\,rank}
\opn\superheight{superheight}\opn\lcm{lcm}
\opn\trdeg{tr\,deg}
\opn\reg{reg} \opn\lreg{lreg} \opn\ini{in} \opn\lpd{lpd}
\opn\size{size}
\opn\div{div} \opn\Div{Div} \opn\cl{cl} \opn\Cl{Cl}
\opn\Spec{Spec} \opn\Supp{Supp} \opn\supp{supp} \opn\Sing{Sing}
\opn\Ass{Ass} \opn\Min{Min}
\opn\Ann{Ann} \opn\Rad{Rad} \opn\Soc{Soc}
\opn\Im{Im} \opn\Ker{Ker} \opn\Coker{Coker} \opn\Am{Am}
\opn\Hom{Hom} \opn\Tor{Tor} \opn\Ext{Ext} \opn\End{End}
\opn\Aut{Aut} \opn\id{id}  \opn\deg{deg}

\opn\nat{nat}
\opn\pff{pf}
\opn\Pf{Pf} \opn\GL{GL} \opn\SL{SL} \opn\mod{mod} \opn\ord{ord}
\opn\Gin{Gin} \opn\Hilb{Hilb}
\opn\aff{aff} \opn\con{conv} \opn\relint{relint} \opn\st{st}
\opn\lk{lk} \opn\cn{cn} \opn\core{core} \opn\vol{vol}
\opn\link{link} \opn\star{star}
\opn\gr{gr}

\def\pot#1#2{#1[\kern-0.28ex[#2]\kern-0.28ex]}

\opn\dirlim{\underrightarrow{\lim}}
\opn\inivlim{\underleftarrow{\lim}}
\let\to=\rightarrow

\def\Implies{\ifmmode\Longrightarrow \else
        \unskip${}\Longrightarrow{}$\ignorespaces\fi}
\def\implies{\ifmmode\Rightarrow \else
        \unskip${}\Rightarrow{}$\ignorespaces\fi}
\def\iff{\ifmmode\Longleftrightarrow \else
        \unskip${}\Longleftrightarrow{}$\ignorespaces\fi}

\let\:=\colon
\newtheorem{Theorem}{Theorem}[]
\newtheorem{Lemma}[Theorem]{Lemma}
\newtheorem{Corollary}[Theorem]{Corollary}
\newtheorem{Proposition}[Theorem]{Proposition}

\theoremstyle{definition}

\newtheorem{Remark}[Theorem]{Remark}

\newtheoremstyle{subsection-tweak}
   {11pt}
   {3pt}%
   {}
   {}%
   {\bfseries}
   {}%
   {.5em}
   {\thmnumber{\@{#1}{}\@{#2}.}%
    \thmnote{~{\bfseries#3.}}}    

\newcounter{numberingbase}

\theoremstyle{subsection-tweak}
\newtheorem{bpp}[Theorem]{}
\newtheorem{bppt}[numberingbase]{}
\newcommand{\bbpp}{\begin{bpp}}
\newcommand{\eepp}{\end{bpp}}
\newcommand{\bbppt}{\begin{bppt}}
\newcommand{\eeppt}{\end{bppt}}

\theoremstyle{theorem}

\theoremstyle{definition}

\newcommand{\val}{\mathrm{val}}		

\let\epsilon\varepsilon
\let\phi=\varphi
\def\qed{\ifhmode\textqed\fi
      \ifmmode\ifinner\quad\qedsymbol\else\dispqed\fi\fi}
\def\textqed{\unskip\nobreak\penalty50
       \hskip2em\hbox{}\nobreak\hfil\qedsymbol
       \parfillskip=0pt \finalhyphendemerits=0}
\def\dispqed{\rlap{\qquad\qedsymbol}}

\opn\dis{dis}
\def\pnt{{\raise0.5mm\hbox{\large\bf.}}}

\opn\Lex{Lex}

\begin{document}

\title{Immediate extensions of valuation rings and ultrapowers}

\author{Dorin Popescu}
\address{Simion Stoilow Institute of Mathematics of the Romanian Academy, Research unit 5, University of Bucharest, P.O. Box 1-764, Bucharest 014700, Romania}

\maketitle

\begin{abstract}
We describe the immediate extensions of a one dimensional valuation ring $V$ which could be embedded in some separation of a ultrapower of $V$ with respect to a certain ultrafilter. For such extensions holds a kind of Artin's approximation. 

 {\it Key words } : immediate extensions,pseudo convergent sequences, pseudo limits, ultrapowers, smooth morphisms, Artin approximation.   \\
 {\it 2010 Mathematics Subject Classification: Primary 13F30, Secondary 13A18,13L05,13B40.}
\end{abstract}

Let $(R,\mm)$ be a Noetherian local ring and $\tilde R$ the ultrapower of $R$ with respect of a non principal ultrafilter on $\bf N$. Then ${\bar R}={\tilde R}/\cap_{n\in {\bf N}} \mm^n{\tilde R}$ is a Noetherian complete local ring which is flat over $R$ (see \cite[Proposition 2.9]{P0}, or \cite[Theorem 2.5]{P1}). Here we try to find an
analogue result in the frame of valuation rings.

Let $V$ be  a  valuation ring  with value group $\Gamma$ containing its residue field $k$, $K$ its fraction field and ${\tilde V}= \Pi_{\mathcal{U}}V$ the ultrapower of $V$ with respect to an ultrafilter $\mathcal{U}$ on a set $U$  (see \cite{CK}, \cite{Scho}, \cite{ADH}). Then 
$${\bar V}={\tilde V}/\cap_{z\in V, z\not =0} z{\tilde V}$$  
is a valuation ring extending $V$, a kind of separation of $\tilde V$. Indeed, $q=\cap_{z\in V, z\not =0} z{\tilde V}$ is a prime ideal because if $x_1x_2\in q$ for some $x_i\in {\tilde V}$ then $\val(x_1x_2)\geq \gamma $ for all $\gamma\in \Gamma$ and so one of $\val(x_i)$
$i=1,2$ must be bigger than all $\gamma\in \Gamma$, that is one of $x_i$ belongs to $q$.

The goal of this paper is to describe the valuation subrings of $\bar V$ (given for some special ultrafilters), which are immediate extensions of $V$ when $\dim V=1$.  If the characteristic of $V$ is $>0$ then there exist some immediate extensions which cannot be embedded in $\bar V$ (see Remark \ref{r2}).

We owe thanks to the referee who hinted us some mistakes in a preliminary version of this paper. 

An inclusion $V \subset V'$ of valuation rings is an \emph{immediate extension} if it is local as a map of local rings and induces isomorphisms between the value groups and the residue fields of $V$ and $V'$.  
$V$ has some maximal immediate extensions (see \cite{Kap}). If the characteristic of the residue field of $V$ is zero then there exists an unique maximal immediate extension of $V$.

Let $\lambda$ be a fixed limit ordinal  and $v=\{v_i \}_{i < \lambda}$ a sequence of elements in $V$ indexed by the ordinals $i$ less than  $\lambda$. Then $v$ is \emph{pseudo convergent} if 

$\val(v_{i} - v_{i''} ) < \val(v_{i'} - v_{i''} )    \ \ \mbox{(that is,} \ \ \val(v_{i} - v_{i'}) < \val(v_{i'} - v_{i''})  ) \ \ \mbox{for} \ \ i < i' < i'' < \lambda$
(see \cite{Kap}, \cite{Sch}).
A  \emph{pseudo limit} of $v$  is an element $z \in V$ with 

$ \val(z - v_{i}) < \val(z - v_{i'}) \ \ \mbox{(that is,} \ \ \val(z -  v_{i}) = \val(v_{i} - v_{i'})) \ \ \mbox{for} \ \ i < i' < \lambda$. We say that $v$  is 
\begin{enumerate}
\item
\emph{algebraic} if some $f \in V[T]$ satisfies $\val(f(v_{i})) < \val(f(v_{i'}))$ for large enough $ i < i' < \lambda$;

\item
\emph{transcendental} if each $f \in V[T]$ satisfies $\val(f(v_{i})) = \val(f(v_{i'}))$ for large enough $i < i' < \lambda$,

\item \emph{fundamental} if for any $\gamma\in \Gamma$  there exist $i<i'$  large enough such that  $\val(v_i-v_{i'})>\gamma$, $\Gamma$ being the value group of $V$.
\end{enumerate}

We need \cite[Proposition A.6]{P}, which is obtained using \cite[Theorem 6.1.4]{CK} and says in particular the following:

\begin{Proposition}\label{kes} 
Let $U$ be an infinite set with  $card\ U=\tau$.   Then there exists an ultrafilter $\mathcal{U}$ on  $U$  such that for any valuation ring $V$ any system of polynomial equations
$(g_i((X_j)_{j\in J})_{i\in I}$ with $card\ I\leq \tau$ in variables $(X_j)_{j\in J}$ with coefficients in  the ultrapower ${\tilde V}=\Pi_{\mathcal{U}}V$ has a solution in  ${\tilde V}$ if and only if all its 
finite subsystems have.
\end{Proposition}
The above proposition is trivial when $U={\bf N}$. In general  the ultrafilter $\mathcal U$ is very special given by  \cite[Theorem 6.1.4]{CK}.

\begin{Lemma} \label{p1} Let $U$ be an infinite set with $card\ U=\tau$, $V$ a valuation ring with value group $\Gamma$ and $card\ \Gamma\leq \tau$ and   $\lambda$ be an ordinal with $card\ \lambda\leq \tau$. Let  $\mathcal{U}$ be the ultrafilter on  $U$ given by the above proposition, ${\tilde V}= \Pi_{\mathcal{U}}V$ the ultrafilter of $V$ with respect to $\mathcal U$ and $\bar V$ its separation introduced above. Then any pseudo convergent sequence  ${\bar v}=({\bar v}_i)_{i<\lambda}$ over  $ V$  has a pseudo limit in $\bar V$.
\end{Lemma}
\begin{proof}

Let $\mathcal{S}$ be the system of polynomial equations over $\bar V$
$$S_i:=X-{\bar v}_i-Y_i({\bar v}_{i+1}-{\bar v}_i); Y_iY'_i-1, \ i<\lambda.$$
  For each $\gamma\in \Gamma_+$ choose an element $z_{\gamma}\in V$ with $\val(z_{\gamma})=\gamma$ and lift ${\bar v}_i$ to some elements ${\tilde v}_i\in {\tilde V}$. Let $\mathcal{S}'$ be the system of polynomial equations
 $$S'_{i\gamma}:=X-{\tilde v}_i-Y_i({\tilde v}_{i+1}-{\tilde v}_i)-z_{\gamma}Z_{\gamma};$$ 
$$Y_iY'_i-1,\ \ \mbox{for}\ \ i<\lambda,\ \gamma\in \Gamma_+,$$
and some variables $X,Y_i,Y'_i,Z_{\gamma}$. 

Then $\mathcal{S}'$ has a solution in $\tilde V$ if and only if $\mathcal{S}$ has a solution modulo $z_{\gamma}{\tilde V}$ for all $\gamma\in \Gamma$, that is if $\mathcal S$ has a solution in $\bar V$, which happens if and only if  $({\tilde v}_i)_{i<\lambda}$ has a pseudo limit in $\tilde V$. Note that the cardinal of the system $\mathcal{S}'$ is   $ \leq \tau$.
By the above proposition, $\mathcal{S}'$ has solutions in $\tilde V$ if and only if every finite subsystem $\mathcal{T}$ of $\mathcal{S}'$ has a solution in $\tilde V$. We may enlarge $\mathcal{T}$ such that it has the form 
$$(S'_{i\gamma})_{i=i_1,\ldots,i_e;\gamma=\gamma_1,\ldots,\gamma_e}$$ for some $i_1<\ldots <i_e<\lambda$ and $\gamma_1,\ldots, \gamma_e\in \Gamma_+$ . But then $x={\tilde v}_{i_e+1}$ induces  a solution of $\mathcal{T}$ in $\tilde V$ because 
$$\val({\tilde v}_{ i_e+1}-{\tilde v}_{i_j})=\val({\tilde v}_{i_{j+1}}-{\tilde v}_{i_j})$$
 for $1\leq j\leq e$ and so there exist some units $y_j\in {\tilde V}$ such that
 $${\tilde v}_{ i_e+1}-{\tilde v}_{i_j}-y_j({\tilde v}_{i_{j+1}}-{\tilde v}_{i_j})\in \cap_{z\in V, z\not =0} z{\tilde V},$$
 for $1\leq j\leq e$.  Thus $({\bar v}_i)_{i<\lambda}$ has a pseudo limit $x$ in $\bar  V$. 
\hfill\ \end{proof}
\begin{Remark} Let $K$ be the fraction field of $V$. If in the above proof $v$  is transcendental then $\val(x)\in \Gamma$ and even the extension $K\subset K(x)$ is immediate (see \cite[Theorem 2]{Kap}). If $v$ is algebraic then $\val( x)$ could be in ${\tilde \Gamma}\setminus \Gamma$, $\tilde \Gamma$ being  the value group of $\tilde V$. 
\end{Remark}
\begin{Lemma} \label{com} Let $U$, $\tau$, $\mathcal U$, $V$, $\Gamma$ be as in Lemma \ref{p1}.  Then the extension $V\subset {\bar V}$ factors through the completion of $V$. 
\end{Lemma}
\begin{proof} By the above lemma any fundamental sequence over $V$ has a limit in $\bar V$. The limits of the fundamental 
sequences over $V$ form a valuation subring $\hat V$ which must be separate 
 because $\cap_{z\in V, z\not = 0} z{\bar V}=0$. Hence $\hat V$ is the completion of $V$.
\hfill\ \end{proof}

\begin{Lemma} \label{mb}  Let $V, \Gamma, {\mathcal U}, U,\tau {\tilde V},{\bar V}$ be as in Lemma \ref{p1}, $a$ an element of $V$ with $\val(a)>0$ and $B$ a finitely presented $V$-algebra. Assume  $V$ is Henselian and the completion inclusion $V\subset {\hat V}$ is separable. Then any $V$-morphism $B\to {\bar V}$ could be lifted modulo $a{\bar V}$ to a $V$-morphism $B\to {\tilde V}$.
\end{Lemma}
\begin{proof} The proof is similar to the proof of \cite[Corollary 2.7]{P0} or part of the proof of \cite[Theorem 2.9]{P1}
(see also \cite{BDLV}). Let $B\cong V[Y]/(f)$, $Y=(Y_1,\ldots,Y_n)$, $f=(f_1,\ldots,f_m)$ and ${\bar w}:B\to {\bar V}$ given by $Y\to {\bar y}\in {\hat V}^n$, let us say that $\bar y$ is induced by ${\tilde y}=[(y_u)_{u\in U}]\in \tilde V$. Set $\gamma=\val(a)$. By \cite[Theorem 1.2]{MB} applied to $V$, there exist a positive integer $N$ and $\nu\in \Gamma_+$ such that if $z\in V$ and 
 $\val(f(z))\geq N\gamma+\nu$
  then there exists $z'\in 
V$ such that $f(z')=0$ and $\val(z-z')\geq \gamma$.

By construction we have in particular $\val(f((y_u))\geq N\gamma+\nu$ for all $u$ from a set $\delta\in \mathcal{U}$. So there exists $y'_u\in V$ such that $f(y'_u)=0$ and $\val(y_u-y'_u)>\gamma$. Define $y'_t=0$ if $t\not \in \delta$ and let ${\tilde y}'=[(y'_u)_u\in {\tilde V}]$. Then $f({\tilde y}')=0$ in $\tilde V$ and the $V$-morphism $B\to {\tilde V}$ given by $Y\to {\tilde y}'$ lifts $\bar w$ modulo $a{\hat V}$. 
\hfill\ \end{proof}

\begin{Lemma} \label{l} Let $V, \Gamma, {\mathcal U}, U,\tau {\tilde V},{\bar V}$ be as in Lemma \ref{p1} and $V'\subset {\bar V}$ a valuation subring, which is an immediate extension of $V$. Assume  $V$ is Henselian and the completion inclusion $V\subset {\hat V}$ is separable. Then any algebraic pseudo convergent sequence of $V$ which has a pseudo limit in $V'$ has one also in $V$. 
\end{Lemma}

\begin{proof} Let $v=(v_j)_{j<\lambda}$ be an algebraic pseudo convergent sequence of $V$ which has a pseudo limit $x$ in $V'$. Let $h\in V[X]$ be a polynomial of minimal degree among  the polynomials $f\in V[Y]$ such that
 $\val(f(v_i))<\val(f(v_j))$ for large $i<j<\lambda$. Set $h^{(i)}=\partial^ih/\partial X^i$, $0\leq i\leq \deg h$ with $h^{(i)}\not =0$. By  \cite[Proposition 6.5]{Po0} there exists an ordinal $\nu<\lambda$ such that  
 an element $z\in V$ is a pseudo limit of $v$ if and only if $\val(h^{(i)}(z))=\val(h^{(i)}(v_{\nu}))$ where $h^{(i)}\not =0$ and $\val(h(z))>\val(h(v_{\rho}))$ for $\nu\leq \rho<\lambda$. In particular, we have $\val(h^{(i)}(x))=\val(h^{(i)}(v_{\nu}))$ where $h^{(i)}\not =0$ and $\val(h(x))>\val(h(v_{\rho}))$ for $\nu\leq \rho<\lambda$. Thus 
if $z\in V$ satisfies 
  $\val(h^{(i)}(z))= \val(h^{(i)}(x))$ for all $0\leq i\leq \deg h$  with $h^{(i)}\not =0$ then $z$ is a pseudo limit of  $v$.
  
  Let $d_i\in V$ such that $\val(d_i)=\val(h^{(i)}(x))$ for  $0\leq i\leq \deg h$ with $h^{(i)}\not =0$,
  let us say $h^{(i)}(x)=d_it_i$ for some invertible $t_i\in V'$,  and  $g$  the system of equations $h^{(i)}(Z)-d_iU_i$, $U_iU'_i-1$. If $z, (u_i)_i$ is a solution of $g$ in $V$ then $z$ is a pseudo limit of  $(v_j)_{j<\lambda}$.   But the map $B:=V[Z,(U_i)_i]/(g)\to {\bar V}$ given by $(Z, (U_i), (U_i'))\to (x,(t_i),(t_i^{-1}))$ could be lifted by Lemma \ref{mb} to a map $B\to {\tilde V}$, that is $g$ has a solution in $\tilde V$ and so in $V$ as well. This ends the proof.
\hfill\ \end{proof} 
 \begin{Remark}\label{r} If $V$ is Henselian and $V'$ is a filtered direct limit of smooth $V$-algebras we get as above  that any algebraic pseudo convergent sequence of $V$ which has a pseudo limit in $V'$ has also one in $V$. Indeed, let $x,(v_j)_{j<\lambda},h,(d_i), (t_i),g$ as above. Then the solution $(x,(t_i),(t_i^{-1}))$ of $g$ in $V'$ comes from a solution of $g$ in a smooth $V$-algebra $C$. But there exists a $V$-morphism   $\rho$ from $C$ to  $V$ because $V$ is Henselian. Thus we get a solution of $g$ in $ V$ via $\rho$, so $(v_j)_{j<\lambda}$ has a pseudo limit  in $V$.
 \end{Remark}

\begin{Lemma} \label{l1} Let $V, \Gamma, {\mathcal U}, U,\tau {\tilde V},{\bar V}$ be as in Lemma \ref{p1} and $V''\subset V'\subset {\bar V}$ some valuation subrings such that $V\subset V''$, $V''\subset V'$ are immediate extensions.  Assume  $V''$ is Henselian and the completion inclusion $V''\subset {\hat V}''$ is separable. Then any algebraic pseudo convergent sequence of $V''$ which has a pseudo limit in $V'$ has one also in $V''$. 
\end{Lemma}
\begin{proof} Let $\bar V''$ be given from $V$ as $\bar V$ from $V$ and $(v_j)_j$ an algebraic pseudo convergent sequence over $V''$ which has a pseudo limit in $V'\subset {\bar V}\subset {\bar V''}$. By Lemma \ref{l} applied to $V''$ it has one in $V''$. Note that the construction of $\bar V$, ${\bar V}''$ are done with the same $U$ and $\mathcal{U}$.
\hfill\ \end{proof} 

\begin{Theorem}\label{t} Let $V\subset V'$ be an immediate extension of one dimensional valuation rings and $ \Gamma, {\mathcal U}, U,\tau, {\tilde V},$
${\bar V}$ be as in Lemma \ref{p1}. Assume  $V'$ is complete and card $U\geq $card $ \Gamma$. The following statements are equivalents:
\begin{enumerate}
\item the extension $V\subset {\bar V}$ factors through $V'$,

\item for any valuation subring $V''\subset  V'$ such that $V\subset V''$ and $V''\subset V'$ are immediate extensions any algebraic pseudo convergent sequence of $V''$ which is not fundamental and has a pseudo limit in $V'$ has one also in $V''$.
\end{enumerate}
Moreover, if $V\subset V'$ is separable and one from $(1)$, $(2)$ holds then  $V'$ is a filtered direct limit of smooth $V$-algebras.
\end{Theorem}

\begin{proof}    Suppose $(1)$ holds and let $V''$ be as in $(2)$. Then the completion ${\hat V}''$ of $V''$ is contained in 
${\bar V}''$ by Lemma \ref{com}. An algebraic pseudo convergent sequence $v$ over $V''$ which has a pseudo limit in $V'\subset {\bar V}\subset {\bar V''}$ must have a pseudo limit in ${\hat V}''$ by Lemma \ref{l1} because ${\hat V}''$ is Henselian since $\dim V''=1$. Then $v$ has a pseudo limit in $V''$ too by \cite[Lemma 2.5]{Po0} for example.

 Assume  $(2)$ holds.
 Let $V''\subset V'$ be a valuation subring such that $V\subset V''$ and $V''\subset V'$ are immediate and $K''$ its fraction field. Applying Zorn's Lemma we may suppose that $V''$ is maximal for inclusion among those immediate extensions $W\subset V'$ of $V$ such that $V\subset {\bar V}$ factors through $W$.   Assume that $V''\not =V'$.  Let  $x\in V'\setminus V''$ and $v$ be a pseudo convergent sequence over $V''$ having $x$ as a pseudo limit  but with no pseudo limit in $V''$ (see \cite[Theorem 1]{Kap}). Then $v$ is either fundamental or transcendental by $(2)$. If $v$ is transcendental then   $K''\subset K''(x)$ is the extension constructed in \cite[Theorem 2]{Kap} for $v$. By Lemma \ref{p1} we see that $v $ has a pseudo limit $z$ in ${\bar V}''$.  Actually, the proof of Lemma \ref{l1} gives that $z$ could be taken in the completion  ${\hat V}''$ of $V''$. 
Then  the unicity given by \cite[Theorem 2]{Kap} shows that $K''(x)\cong K''(z)$ and so the extension $V''\subset {\bar V}$ factors through $V_1=V'\cap K''(x)$ because
 $z\in {\hat V}''\subset V'\subset {\bar V}$ since $V'$ is complete.
If $v$ is fundamental then as above $x\in {\hat V}''\subset {\bar V}$. In both cases the extension $V''\subset {\bar V}$ factors through $V_1=V'\cap K(x)$. These
 contradict that $V''$ is maximal by inclusion, that is  $V''$ must be $V'$.

 Now suppose  $V\subset V'$ is separable and  $(2)$ holds.
 We reduce to the case when the fraction field extension $K\subset K'$ of $V\subset V'$ is of finite type because $V'$ is a filtered direct union of $V'\cap L$ for some subfields $L\subset K'$, which are finite type field extensions of $K$. By induction on the number of generators of $L$ over $K$ we may reduce to the case when $K'=K''(x)$ for some element $x\in V'$, $K''$ being the fraction field  of a valuation subring $V''\subset V'$ as in $(2)$. As $K'/K$ is separable of finite type we may arrange that $K'/K''$ is still separable. Then $x$ is a pseudo limit of a pseudo convergent sequence $v$ from $V''$ which has no pseudo limit in $K''$ (see \cite[Theorem 1]{Kap}). Assume that $v$ is not a fundamental sequence. Then $v$ is transcendental by $(2)$  and so $V'$ is a filtered direct union of localizations of polynomial $V''$-subalgebras of $V'$ in one variable by \cite[Theorem 3.2]{Po} (see also \cite[Lemma 15]{P}). If $v$ is fundamental then $V''\subset V'$ is dense and separable and we may apply a theorem of type N\'eron-Schappacher (see e.g. \cite[Theorem 4.1]{Po0}). 
\hfill\ \end{proof}

\begin{Remark}\label{r2} If $V\subset V'$ is the  valuation ring extension given in \cite[Example 3.13]{Po0} then $V'$ is not a filtered direct limit of smooth $V$-algebras (see \cite[Remark 6.10]{Po0}) and so cannot be embedded in $\bar V$ by Theorem \ref{t}.
\end{Remark}
The following corollary is a kind of Artin approximation (see \cite{A}, \cite{P1}, \cite{H}, \cite{R}) in the frame of valuation rings. Its statement extends the idea of \cite[Corollary 8, Theorem 11]{KP} and \cite[Theorem 14]{PR} replacing the order by the valuation.
\begin{Corollary}\label{c1} Let $V \subset V'$ be  a separable immediate extension such that for any valuation subring  $V''\subset  V'$  such that $V\subset V''$ and $V''\subset V'$ are immediate extensions any algebraic pseudo convergent sequence of $V''$ which has a pseudo limit in $V'$ has one also in $V''$. Let $f$ be  a finite system of polynomials equations from $V[Y]$, $Y=(Y_1,\ldots,Y_n)$, which has a solution in $V'$. Assume    $V$ is complete and $\dim V=1$.
 Then $f$ has a solution in $V$. Moreover, if $y'=(y'_1,\ldots,y'_n)$ is a solution of $f$ in $V'$ then there exists a solution $y=(y_1,\ldots,y_n)$ of $f$ in $V$ such that $\val(y_i)=\val(y'_i)$ for $1\leq i\leq n$. 
\end{Corollary}

\begin{proof} Let $y'$ be a solution of $f$ in $V'$, $B=V[Y]/(f)$ and $w:B\to V'$ be the map given by $Y\to y'$. Let $\Gamma,{\mathcal U}, U, {\tilde V},{\bar V}$ be as in Theorem \ref{t}.   Then  the extension $V\subset {\bar V}$ factors through $V'$ and $\bar w$  the composite   map $B\xrightarrow{w} V'\to {\bar V}$ could be lifted to a map ${\tilde w}:B\to {\tilde V}$ by Lemma \ref{mb}. Thus $f$ has in ${\tilde V}$ the solution ${\tilde w}(Y)$ and so it has also a solution in $V$.

Now,  let $a=(a_1,\ldots,a_n)\in V^n$ be such that $\val(a_i)=\val(y'_i)$ for $1\leq i\leq n$. Then there exists an unit $z'_i\in V'$ such that $y'_i=a_iz'_i$ and the system $g$ obtained by adding to $f$ the equations $Y_i-a_iZ_i$, $ Z_iT_i-1$, $1\leq i\leq n$ has in $V'$ the solution $y'$, $z'=(z'_1,\ldots,z'_n)$, $t'=(t'_1,\ldots,t'_n)$, the last ones are given by the inverses of $(z'_i)$. So $g$ has a solution $y,z,t$ in $V$ and it follows that $\val(y_i)=\val(y'_i)$, $1\leq i\leq n$.  
\hfill\ \end{proof}
\begin{Remark} Another proof of the above corollary could be done using that  $V'$ is a filtered direct limit of smooth $V$-algebras (see Theorem \ref{t}).
\end{Remark}


\begin{thebibliography}{99}
\normalsize{
 \bibitem{A} M.\ Artin, {\em  Algebraic approximation of structures over complete local rings}. Pub. Math. Inst. Des.
Hautes. Scientif. {\bf 36(1)},(1969), 23–58.

\bibitem{ADH} M.\ Aschenbrenner, L.\ van den Dries, J.\ van der Hoeven, {\em Asymptotic differential algebra and model theory of
              transseries}, Annals of Mathematics Studies, {\bf 195}, Princeton University Press, Princeton, NJ, 2017.





\bibitem{BDLV} J.\ Becker, J. \ Denef, L.\ Lipshitz, L.\ van der Dries, {\em Ultraproducts and approximation in local rings I}, Inventiones Math., {\bf 51}, (1979), 189-203.

\bibitem{CK} C.\ C.\ Chang, H.\ J.\ Keisler, {\em Model theory}, 3rd ed., Studies in Logic and the Foundations of Mathematics,
vol. 73, North-Holland Publishing Co., Amsterdam, 1990.  

\bibitem{H} H.\ Hauser, {\em The classical Artin approximation theorems}, Bull. Amer.
Math. Soc., {\bf 54}, (2017), 595-633; 

\bibitem{Kap} I.\ Kaplansky, {\em Maximal fields with valuations}, Duke Math. J. {\bf9} (1942), 303-321.

\bibitem{KP} Z.\ Kosar, D.\ Popescu, {\em Nested Artin Strong Approximation Property},  J. of Pure Algebra and Applications, {\bf 222 (4)}, (2018), 818-827.

\bibitem{MB} L.\ Moret-Bailly, {\em An extension of Greenberg's theorem to general valuation rings}, Manuscripta Math.
{\bf 139} (2012), no. 1-2, 153-166.


\bibitem{O} A.\  Ostrowski, {\em Untersuchungen zur arithmetischen Theorie der K\"orper}, Math. Z. {\bf39} (1935), no. 1,
321-404.
\bibitem{P0}  D.\ Popescu, {\em Algebraically pure morphisms}, Rev. Roum. Math. Pures et Appl., {\bf 24}, (1979), 947-977.

 \bibitem{Po} D.\ Popescu, {\em On Zariski's uniformization theorem}, in Algebraic geometry, Bucharest 1982 (Bucharest,
1982), Lecture Notes in Math., vol. 1056, Springer, Berlin, 1984,  264-296.

\bibitem{Po0} D.\ Popescu, {\em Algebraic extensions of valued fields}, J. Algebra {\bf 108}, (1987),
513-533.
 
 \bibitem{P1}  D.\ Popescu, {\em   Artin approximation}, in Handbook of Algebra, vol. {\bf 2}, Ed. M.Hazewinkel,
(2000), Elsevier Science, 321-356.
 
  \bibitem{P} D.\ Popescu, {\em N\'eron desingularization of extensions of valuation rings with an Appendix by K\k{e}stutis \v{C}esnavi\v{c}ius},  to appear 
in Proceedings of Transient Transcendence in Transylvania 2019, Eds. Alin Bostan, Kilian
Raschel (possible in Springer Collection PROMS),
  arxiv/AC:1910.09123v4.
 
\bibitem{PR} D.\ Popescu, G.\ Rond, {\em  Remarks on Artin Approximation with constraints}, Osaka J. Math. {\bf 56},
(2019), 431–440. 

\bibitem{R} G.\ Rond, {\em Artin Approximation}, J. of Singularities, {\bf 17}, (2018), 108-192.

\bibitem{Sch} O.\ F.\ G.\ Schilling, {\em The theory of valuations}, Mathematical Surveys, Number IV, American Math. Soc., (1950).

\bibitem{Scho} H.\ Schoutens, {\em The use of Ultraproducts in Commutative Algebra}, Lect. Notes in Math., Springer, 1999.
}
\end{thebibliography}
\end{document}